\documentclass[12pt]{article}
\usepackage[english]{babel}
\usepackage[utf8]{inputenc} 
\usepackage[T1]{fontenc} 
\usepackage{lmodern} 
\usepackage{amsmath}
\usepackage{amssymb}
\usepackage{amsthm}
\usepackage{hyperref}
\usepackage{titlesec}
\usepackage[margin=3cm]{geometry}

\titleformat{\section}[block]{\Large\bfseries}{\thesection.}{3pt}{} 
\titleformat{\subsection}[runin]{\normalfont\bfseries}{\thesubsection.}{3pt}{}[.] 

\newcommand\IN{\mathbb{N}} 
\newcommand\IR{\mathbb{R}} 
\newcommand\IE{\mathbb{E}} 
\newcommand\IP{\mathbb{P}} 
\newcommand\EE{\mathbf{E}} 
\newcommand\PP{\mathbf{P}} 
\newcommand\ind{\mathbf{1}} 
\newcommand\law{\mathcal{L}}
\newcommand\tL{\tilde{L}}
\newcommand\tR{\tilde{R}}
\newcommand\tl{\tilde{l}}
\newcommand\tr{\tilde{r}}
\newcommand\ii{\mathbf{i}}
\newcommand\aaa{\mathbf{a}}
\newcommand\bb{\mathbf{b}}
\newcommand\cc{\mathbf{c}}
\newcommand\dd{\mathbf{d}}

\newtheorem{theorem}{Theorem}
\newtheorem{lemma}[theorem]{Lemma}
\newtheorem{corollary}[theorem]{Corollary}
\newtheorem{proposition}[theorem]{Proposition}
\theoremstyle{definition}
\newtheorem{remark}[theorem]{Remark}

\begin{document}

\title{Particle system approach to wealth redistribution}
\author{Roberto Cortez\footnote{
CIMFAV, Facultad de Ingenier\'ia, Universidad de Valparaíso, General Cruz 222, Valpara\'iso, Chile.
E-mail: \texttt{rcortez@dim.uchile.cl}. Supported by Fondecyt Postdoctoral Project 3160250 and by Iniciativa Cient\'ifica Milenio NC120062.
}}
\maketitle

\begin{abstract}
We study a stochastic $N$-particle system representing economic agents in a population randomly exchanging their money, which is associated to a class of one-dimensional kinetic equations modelling the evolution of the distribution of  wealth in a simple market economy, introduced by Matthes and Toscani \cite{matthes-toscani2008}. We show that, unless the economic exchanges satisfy some exact conservation condition, the $p$-moments of the particles diverge with time for all $p>1$, and converge to 0 for $0<p<1$. This establishes a qualitative difference with the kinetic equation, whose solution is known to have bounded $p$-moments, for all $p$ smaller than the Pareto index of the equilibrium distribution. On the other hand, the case of strictly conservative economies is fully treated: using probabilistic coupling techniques, we obtain stability results for the particle system, such as propagation of moments, exponential equilibration, and uniform (in time) propagation of chaos with explicit rate of order $N^{-1/3}$.
\end{abstract}



\section{Introduction}

\subsection{Kinetic equation for wealth redistribution}

In the last decades, the ideas and techinques of the classical kinetic theory of dilute gases have been successfully applied to the study of wealth redistribution in a large population, as part of a discipline known as \emph{econophysics} \cite{chakraborti-chakrabarti2000,chatterjee-chakrabarti-manna2004,cordier-pareschi-toscani2005,dragulescu-yakovenko2000,during-matthes-toscani2008,matthes-toscani2008}. In this context, physical particles are replaced by economic agents, and binary collisions are replaced by economic exchanges between them (trades). Typically, an agent is characterized by her wealth $v \in \IR$, and the evolution of the probability distribution of wealth among the population $f_t(dv)$ solves the following kinetic-type equation:
\begin{equation}
\label{eq:kinetic}
\partial_t f_t + f_t = Q_+ (f_t,f_t),
\end{equation}
where $Q_+ (f,f)$ is the measure defined by
\begin{equation}
\label{eq:Q+}
\int_{\IR} \varphi(v) Q_+(f,f)(dv)
:= \frac{1}{2} \iint_{\IR^2} \EE[\varphi(v') + \varphi(v_*') ] f(dv) f(dv_*).
\end{equation}
Here, $v,v_*\in \IR$ represent the riches of two agents in the population prior to the trade, and the post-trade riches $v',v_*' \in \IR$ are given by the rule
\begin{equation}
\label{eq:inter_rule}
(v, v_*) \mapsto (v', v_*') = (Lv + Rv_*, \tilde{L}v_* + \tilde{R}v),
\end{equation}
where $(L, R, \tilde{L}, \tilde{R})$ is some random vector on $\IR^4$ with known distribution and $\EE$ denotes the expectation with respect to it. In this form, the model \eqref{eq:kinetic}-\eqref{eq:inter_rule} was introduced by Matthes and Toscani \cite{matthes-toscani2008}, and it can be seen as a generalization of Kac's one dimensional toy model of the spatially homogeneous Boltzmann equation \cite{kac1956}, where $L = \cos \theta = \tilde{L}$ and $R = -\sin \theta = -\tilde{R}$ for $\theta$ uniformly chosen on $[0,2\pi)$; thus, the interactions preserve the kinetic energy in this case: $v'^2 + v_*'^2 = v^2 + v_*^2$. In general, the wealth of an agent can be any real number $v\in\IR$, interpreting $v<0$ as debt. However, we will assume that $f_0$ is concentrated on $\IR_+$ and that $L,\tilde{L},R,\tilde{R} \geq 0$ a.s., which implies that $f_t$ is also concentrated on $\IR_+$, thus excluding debts.

One of the main features of the model \eqref{eq:kinetic}-\eqref{eq:inter_rule} is that it typically admits a unique equilibrium distribution $f_\infty$. More importantly: one looks for conditions ensuring that $f_\infty$ has a heavy tail, i.e., to determine if there exists some $\alpha>1$, known as the \emph{Pareto index}, such that the moments $\int v^p f_\infty(dv)$ are finite for all $0<p<\alpha$ and infinite for $p>\alpha$. This is crucial to assess the validity of the model, since the empirical data shows that all real economies exhibit a Pareto tail of some index $\alpha>1$: economies with low inequality are associated with large $\alpha$, and viceversa.

In analogy with the preservation of energy of Kac's model, earlier versions of \eqref{eq:inter_rule}  (see for instance \cite{chakraborti-chakrabarti2000}) assumed \emph{exact preservation of wealth}, i.e., $v' + v_*' = v + v_*$ for all $v,v_*$, which in terms of the interaction coefficients corresponds to
\begin{equation}
\label{eq:as_preservation}
L + \tR = 1 = \tL + R \qquad \text{a.s.}
\end{equation}
In this case, we say that the economy is \emph{strictly conservative}. However, it can be shown that this condition necessarily leads to an equilibrium distribution with slim tails (formally, $\alpha=\infty$), thus rendering the model unrealistic from the economic point of view. This fact led to the development of a broader class of models \cite{cordier-pareschi-toscani2005,during-matthes-toscani2008,matthes-toscani2008}, where the interactions between agents have an intrinsic risk which can produce a gain or loss of total wealth in each exchange, but still preserving wealth \emph{in the mean}, that is, one drops \eqref{eq:as_preservation} in favor of
\begin{equation}
\label{eq:mean_conservation}
\EE[L + \tR] = 1 = \EE[\tL + R],
\end{equation}
which still implies that the mean wealth $\int v f_t(dv)$ preserves its initial value. Under this weaker condition, in \cite{matthes-toscani2008} it is shown that $f_\infty$ can indeed exhibit a heavy tail with Pareto index depending explicitly on the moments of the interaction coefficients $L$, $R$, $\tL$ and $\tR$, and is given by \begin{equation}
\label{eq:pareto_index}
\alpha = \sup \{ p \geq 1: \EE[L^p + R^p + \tL^p + \tR^p] < 2 \},
\end{equation}
assuming some non-degeneracy condition (such as \eqref{eq:LRnot01}). Similarly, it is also shown that the moments $\int v^p f_t(dv)$ stay bounded for $p<\alpha$ and diverge with time for $p>\alpha$.

\subsection{Particle system} 

Formally, the kinetic equation \eqref{eq:kinetic} represents the evolution of the distribution of wealth in an \emph{infinite} population. As it usually done with Kac's model and the Boltzmann equation, one can associate with \eqref{eq:kinetic} a \emph{finite} stochastic $N$-particle system, which should give a more transparent link between the interactions at the level of the agents and the behaviour of the whole ensemble. This particle system is a Markov pure-jump process on $\IR^N$, which we describe as follows: interactions between agents or ``particles'' take place at random times, with time intervals having exponential distribution with rate $N/2$, and at each interaction two distinct agents are selected at random, and their riches are then updated according to the rule \eqref{eq:inter_rule}. The vector of $N$ initial riches is chosen following a prescribed symmetric distribution, and all previous random choices are made independently. This description unambiguously specifies (the law of) the particle system, which we denote\footnote{In our notation, we do not make explicit the dependence of the particle system on the number of particles $N$.} $(\textbf{V}_t)_{t\geq 0} = (V_t^1,\ldots,V_t^N)_{t\geq 0}$. See \eqref{eq:SDEps} below for an explicit definition using an SDE.

There are two main motivations to introduce such a particle system. The first one is numerical approximation: while \eqref{eq:kinetic} typically can not be solved explicitly, it is straightforward to simulate the particle system $(\textbf{V}_t)_{t\geq 0}$ even for $N$ relatively large, and one expects the empirical distribution $\frac{1}{N}\sum_i \delta_{V_t^i}$ to be a good approximation of $f_t$; in fact, the first works in this context were purely numerical \cite{chakraborti-chakrabarti2000,chatterjee-chakrabarti-manna2004,dragulescu-yakovenko2000}. The second motivation is mathematical validation: one would like to prove mathematically that the kinetic equation is indeed the limit as $N\to\infty$, in some sense, of the $N$-particle system, a property known as \emph{propagation of chaos}. For Kac's model, it was obtained by Kac himself in his original paper \cite{kac1956}; since then, numerous propagation of chaos results were obtained by several authors for some related physical models, including the spatially homogeneous Boltzmann equation, see for instance \cite{grunbaum1971,mckean1967,sznitman1984}, ultimately leading to \emph{quantitative} rates of convergence with explicit dependence on $N$ and \emph{uniformly on time}, see \cite{cortez-fontbona2018,mischler-mouhot2013}. This time-uniformity means that, for $N$ large, the kinetic equation is a good approximation of the finite system even for very large times, implying that the stationary distribution of the equation does indeed represent a physical ensemble of particles in thermodynamical equilibrium.

\subsection{Relevant questions}

It is worth mentioning that, although particle systems are widely used as a simulation tool, there are almost no works in the literature that study particle systems in the context of wealth redistribution from a mathematical point of view. Thus, many of its properties remain to be investigated. In the present paper, we are particularly interested in determining if and how some of the relevant properties of the kinetic equation \eqref{eq:kinetic} are transferred to the finite system, hopefully uniformly in $t$ and/or $N$. Time uniformity becomes especially important because, while in the physical context the number of particles is of order of Avogadro's number (thus, any property that does not hold uniformly on time in the finite system may be compensated by the overwhelmingly huge number of physical particles), in the econophysical context the number of agents in a real economy is only in the order of millions, so the desired property could degenerate not so slowly with time in the finite system.

Since the formation of heavy tails is a key feature of the model \eqref{eq:kinetic}-\eqref{eq:inter_rule}, we are thus interested in studying the evolution of the moments at the level of the finite system and see how it relates to the corresponding property of $(f_t)_{t\geq 0}$. Similarly, we would also want to determine the existence of an equilibrium distribution for $\mathbf{V}_t$ and its relation with $f_\infty$. In this vein, we raise the following questions:

\begin{itemize}
\item[(Q1)] If $\alpha>1$ is the Pareto index associated to the model \eqref{eq:kinetic}-\eqref{eq:inter_rule} and for $1<p\neq\alpha$, does a particle in the system have finite moments of order $p$ uniformly on time if and only if $p<\alpha$?

\item[(Q2)] Does the particle system exhibit a non-trivial equilibrium distribution? If the answer is affirmative, does it converge to $f_\infty$ as $N\to\infty$ in some sense?

\item[(Q3)] Does the system propagate chaos uniformly on time?
\end{itemize}

The main goal of this paper is to address these questions. Unfortunately, as we shall see, the answer to all of them is negative, \emph{unless the interactions satisfy some kind of exact preservation condition}, such as \eqref{eq:as_preservation}.

More specifically, in Theorem \ref{thm:ps_moments} we prove that when the interactions are not a.s.\ conservative in some sense, the moments of the particles of order $p>1$ diverge with time, while those of order $p<1$ converge to $0$. It is worth noting that, while in the classical physical setting of Kac's model there is always a preserved quantity, namely, the total energy $\sum_i (V_t^i)^2$ a.s.\ preserves its initial value, in the econophysical context the total wealth $\sum_i V_t^i$ (the analogous of the energy) may \emph{not} be preserved if one only assumes condition \eqref{eq:mean_conservation}. This difference turns out to be crucial, and is the main reason behind the odd behaviour of the moments given in Theorem 1. On the other hand, under the stronger condition \eqref{eq:as_preservation} of strict conservation of wealth, we do have $\sum_i V_t^i = \sum_i V_0^i$ a.s., which will allow to deduce nice stability results for the particle system; for instance, we prove uniform (in time) propagation of chaos in Theorem \ref{thm:UPoC}. These (and other) results provide a deeper understanding of the power and limitations of the particle system as an approximation tool for the kinetic equation, both theoretically and numerically.

\subsection{Previous results}

The first works in the present setting concerned primarily numerical studies for strictly conservative economies \cite{chakraborti-chakrabarti2000,chatterjee-chakrabarti-manna2004,dragulescu-yakovenko2000}. The corresponding kinetic Boltzmann-like equation was later introduced and studied for instance in \cite{slanina2004}, while in \cite{cordier-pareschi-toscani2005} risky trades were considered. The general kinetic model \eqref{eq:kinetic}-\eqref{eq:inter_rule} satisfying condition \eqref{eq:mean_conservation} of preservation of wealth only in the mean was introduced and studied by Matthes and Toscani in \cite{matthes-toscani2008}, where the authors prove the main analytical properties of the solution, including the existence of a heavy-tailed equilibrium distribution $f_\infty$ with Pareto index given by \eqref{eq:pareto_index}. Extensions of this model are later considered for instance in \cite{bassetti-ladelli-matthes2011,bassetti-ladelli-toscani2011}.

Regarding the mathematical behaviour of the particle approximation, as mentioned earlier, there are almost no works in the literature that deal with the particle system for wealth redistribution models. Up to our knowledge, the only mathematical study of this kind is due to Cortez and Fontbona \cite{cortez-fontbona2016}, where the authors prove a propagation of chaos result with explicit polynomial rates in $t$ and $N$, although not uniform in $t$. On the other hand, in the physical context of Boltzmann-like equations the literature is quite extensive, with affirmative answers to the analogous of questions (Q1)-(Q3), see for instance \cite{carrapatoso2015,cortez-fontbona2018,mischler-mouhot2013} for the spatially homogeneous Boltzmann equation and \cite{carlen-carvalho-einav2018,carlen-carvalho-loss2000,cortez2016,hauray2016} for Kac's model.

\subsection{Plan of the paper}

In Section \ref{sec:mean-preserving} we specify our assumptions and notation, give a particular construction of the particle system $\mathbf{V}_t$ suitable for our purposes (using an SDE with respect to a Poisson point measure), and proceed to study the general case of interactions preserving wealth only in the mean (condition \eqref{eq:mean_conservation}). We answer the questions raised before: (Q1) is answered negatively in Theorem \ref{thm:ps_moments} (see also Proposition \ref{prop:EV2}), which will immediately imply that (Q2) and (Q3) also have a negative answer, see Remark \ref{rmk:ps_moments}. Some comments that put these results in perspective are made in Remark \ref{rmk:ps_moments2}. In Section \ref{sec:strictly-conserv} we study the case of strictly conservative economies (i.e., those satisfying \eqref{eq:as_preservation}), proving contractivity in Theorem \ref{thm:contraction}, equilibration in Corollary \ref{cor:equilibration}, propagation of moments in Proposition \ref{prop:moments_prop}, uniform propagation of chaos in Theorem \ref{thm:UPoC}, and convergence of the equilibrium distribution in Corollary \ref{cor:conv_equilibrium}; these results give a positive answer to (Q1)-(Q3) in this case. They are stated in therms of the $2$-Wasserstein distance, and the proofs are based on probabilistic coupling techniques, as the ones used for instance in \cite{cortez2016,cortez-fontbona2016,hauray2016}. We leave the proof of some intermediate technical results for the Appendix.

\section{Mean-preserving interactions}
\label{sec:mean-preserving}

Before we state our results, let us first specify our main assumptions and fix some notation. Throughout this paper, we assume $L,R,\tL,\tR\geq 0$ a.s., and that they satisfy condition \eqref{eq:mean_conservation} of conservation of wealth in the mean. We will also assume that $L$, $R$, $\tL$ and $\tR$ have as many finite moments as the statements of our results require. We assume $f_0$ is a probability distribution concentrated on $\IR_+ = [0,\infty)$ with mean wealth $m := \int v f_0(dv) <\infty$, and we denote $(f_t)_{t\geq 0}$ the collection of probability measures on $\IR_+$ solution to \eqref{eq:kinetic}, which thus satisfies $\int v f_t(dv) = m$ for all $t\geq 0$. We denote $\PP$ and $\EE$ the probability and expectation on the space where $L,R,\tL,\tR$ is defined, while $\law(\cdot)$ denotes the law of a random element.

We now give an explicit construction of the particle system, useful for our purposes. Fix the number of particles $N\in\IN$, and let $\mathcal{P}(dt,dl,dr,d\tl,d\tr,d\xi,d\zeta)$ be a Poisson point measure on $[0,\infty) \times \IR^4 \times [0,N)^2$ with intensity
\begin{equation}
\label{eq:intensityP}
\frac{N}{2} \frac{dt \Lambda(dl,dr,d\tl,d\tr) d\xi d\zeta \ind_{\ii(\xi)\neq \ii(\zeta)}}{N(N-1)}
= \frac{dt \Lambda(dl,dr,d\tl,d\tr) d\xi d\zeta \ind_{\ii(\xi)\neq \ii(\zeta)}}{2(N-1)},
\end{equation}
where $\Lambda := \law(L,R,\tL,\tR)$, and the function $\ii: [0,N) \to \{1,\ldots,N\}$ associates to a continuous variable $\xi \in[0,N)$ the discrete index $\ii(\xi) := \lfloor \xi \rfloor + 1$. In words: the measure $\mathcal{P}$ samples $t$-atoms at rate $N/2$, and for each such $t$ it also samples a realization $(l,r,\tl,\tr)$ of the tuple $(L,R,\tL,\tR)$ and a pair $(\xi,\zeta) \in [0,N)^2$ uniformly at random such that $\ii(\xi) \neq \ii(\zeta)$ (notice that $\int_{[0,N)^2} d\xi d\zeta \ind_{\ii(\xi) \neq \ii(\zeta)} = N(N-1)$). The pair $(\ii(\xi),\ii(\zeta))$ will give the indexes of the particles that interact at each jump. Also, let $\mathbf{V}_0$ be an exchangeable random vector on $\IR^N$ of initial riches with prescribed distribution, independent of $\mathcal{P}$. We denote $\IP$ and $\IE$ the probability expectation on the corresponding probability space.

The particle system $(\textbf{V}_t)_{t\geq 0} = (V_t^1,\ldots,V_t^N)_{t\geq 0}$, is then defined as the solution, starting from $\mathbf{V}_0$, to the stochastic equation
\begin{equation}
\label{eq:SDEps}
d\mathbf{V}_t
= \int_{\IR^4} \int_{[0,N)^2} \sum_{i\neq j} \ind_{\ii(\xi) =i, \ii(\zeta)=j}[\mathbf{V}_{t^-}'^{ij} - \mathbf{V}_{t^-}] \mathcal{P}(dt, dl,dr,d\tl,d\tr, d\xi, d\zeta),
\end{equation}
where the vector $\mathbf{v}'^{ij} \in \IR^N$ corresponds to $\mathbf{v} = (v^1,\ldots,v^N) \in \IR^N$ with its $i$ and $j$ coordinates respectively replaced by $lv^i + rv^j$ and $\tl v^j + \tr v^i$. Since the rate of $\mathcal{P}$ is finite on bounded time intervals, there always exists a unique strong solution of \eqref{eq:SDEps}, and the collection $(V_t^1,\ldots,V_t^N)_{t\geq 0}$ is exchangeable.

We can now state and prove our results. Consider the following condition, which is a weaker version of the exact preservation of wealth \eqref{eq:as_preservation}:
\begin{equation}
\label{eq:weak_as_preservation}
L + R + \tL + \tR = 2 \quad \text{a.s.}
\end{equation}
Also, for $p>0$, call
\begin{align*}
\beta &= \beta_{N,p} := \EE\left[ \left(1+\frac{1}{N}[L+R+\tL+\tR-2]\right)^p \right], \\
\gamma &= \gamma_{N,p} := \frac{N}{2}(1-\beta_{N,p}).
\end{align*}
Note that thanks to assumption \eqref{eq:mean_conservation} and Jensen's inequality, for $0<p<1$ we always have $\beta \leq 1$ and $\gamma \geq 0$. Moreover, we will have $\beta = 1$ and $\gamma = 0$ if and only if condition \eqref{eq:weak_as_preservation} holds. Similarly, for $p>1$ we always have $\beta \geq 1$ and $\gamma \leq 0$, with equality if and only if \eqref{eq:weak_as_preservation} holds.

\begin{theorem}[evolution of moments]
\label{thm:ps_moments}
Let $M_t := \frac{1}{N} \sum_i V_t^i$ be the empirical mean of the particle system. Then, for any $p>0$ fixed and for all $t\geq 0$,
\[
\IE[ (V_t^1)^p ]
\begin{cases}
\leq e^{-\gamma t} \IE M_0^p  & \text{if $p<1$,} \\
\equiv 1 & \text{if $p=1$,} \\
\geq e^{-\gamma t} \IE M_0^p & \text{if $p>1$,} \\
\end{cases}
\]
Consequently, if $\IE M_0^p < \infty$ and if \eqref{eq:weak_as_preservation} does not hold, then for $p<1$ we have $\gamma>0$ and $\lim_{t\to\infty} \IE[ (V_t^1)^p ] = 0$, while for $p>1$ we have $\gamma<0$ and $\lim_{t\to\infty} \IE [ (V_t^1)^p ] = \infty$; also, $\lim_{t\to\infty} V_t^1 = 0$ a.s. Moreover, all these assertions are also true for $M_t$.
\end{theorem}

\begin{proof}
We will prove the desired assertions for $M_t$, and then exchangeability and Jensen's inequality will imply that they also hold for $V_t^1$. We first work in discrete time $n\in\IN$: with a slight abuse of notation (we use the same letters), we call $V_n^i$ the state of particle $i\in\{1,\ldots,N\}$ after $n$ jumps of the particle system, and $M_n = \frac{1}{N} \sum_i V_n^i$. Denote:
\begin{itemize}
\item $(L_n,R_n,\tL_n,\tR_n)_{n\in\IN}$ the interaction coefficients corresponding to each jump, and denote $\IE^C$ the associated expectation. That is, $(L_n,R_n,\tL_n,\tR_n)_{n\in\IN}$ are i.i.d.\ copies of $(L,R,\tL,\tR)$ under $\IE^C$.
\item $(k_n,\ell_n)_{n\in\IN}$ the random indices of the particles that interact at each jump, i.e., each $(k_n,\ell_n)$ is a pair of distinct indices chosen from the set $\{1,\ldots,N\}$ uniformly at random and independently from the rest. Call $\IE_n^I$ the expectation with respect to $(k_n,\ell_n)$ and $\IE^I$ the expectation with respect to the whole collection.
\item $\IE_0$ the expectation with respect to the initial condition $\mathbf{V}_0$.
\end{itemize}
Thus, the global expectation is written as $\IE = \IE_0 \IE^C \IE^I$, and these three measures are independent. The key step of the proof is to define $S_n = \IE^I M_n$, that is, the average of $M_n$ over all possible choices of the indices of particles that interact at each jump, including the jumps prior to $n$. Clearly:
\begin{align*}
M_n
&= \frac{1}{N} \left[ \left(\sum_{i=1}^N V_{n-1}^i \right)
		- V_{n-1}^{k_n} - V_{n-1}^{\ell_n} \right. \\
& \qquad \qquad \qquad \left. \vphantom{\sum_{i=1}^N}		
	 {} + (L_n V_{n-1}^{k_n} + R_n V_{n-1}^{\ell_n})
		+ (\tL_n V_{n-1}^{\ell_n} + \tR_n V_{n-1}^{k_n}) \right] \\
&= M_{n-1}
   + \frac{1}{N} \left[ (L_n + \tR_n - 1)V_{n-1}^{k_n}
                      + (\tL_n + R_n - 1)V_{n-1}^{\ell_n} \right].
\end{align*}
Notice that $\IE^I V_{n-1}^{k_n} = \IE_1^I \cdots \IE_{n-1}^I \IE_n^I V_{n-1}^{k_n} = \IE_1^I \cdots \IE_{n-1}^I M_{n-1} = S_{n-1}$, and also $\IE^I V_{n-1}^{\ell_n} = S_{n-1}$. Defining $K_n = 1 + \frac{1}{N}(L_n+R_n+\tL_n+\tR_n-2)$ and taking $\IE^I(\cdot)$ in the previous equation, we get the recursion $S_n = K_n S_{n-1}$, which gives
\begin{equation}
\label{eq:Sn}
S_n = M_0 \prod_{j=1}^n K_j.
\end{equation}
Now fix $0<p<1$, and notice that $\beta = \IE^C K_n^p$ for all $n$. Thanks to Jensen's inequality, we obtain
\begin{equation}
\label{eq:EMn}
\IE M_n^p
= \IE_0 \IE^C \IE^I M_n^p
\leq \IE_0 \IE^C S_n^p
= \IE_0 \IE^C M_0^p \prod_{j=1}^n K_j^p
= \beta^n \IE M_0^p.
\end{equation}
Going back to continuous time: let $0=\tau_0<\tau_1<\cdots$ be the jump times of the particle system, thus (again abusing notation):
\begin{align*}
\IE M_t^p
= \IE \sum_{n=0}^\infty \ind_{t \in[\tau_n,\tau_{n+1})} M_n^p
&= \sum_{n=0}^\infty \IP(t \in[\tau_n,\tau_{n+1})) \IE M_n^p \\
&\leq \sum_{n=0}^\infty e^{-Nt/2} \frac{(Nt/2)^n}{n!} \beta^n \IE M_0^p,
\end{align*}
thanks to \eqref{eq:EMn} and the fact that the jumps occur at rate $N/2$. This yields $\IE M_t^p \leq e^{-Nt/2} e^{N\beta t/2} \IE M_0^p = e^{-\gamma t} \IE M_0^p$, which concludes with the case $p<1$. The case $p=1$ is trivial, and for $p>1$ the argument follows from \eqref{eq:Sn} exactly as before, with the inequalities reversed.

It remains to prove that $\lim_{t\to\infty} M_t = 0$ a.s.\ when \eqref{eq:weak_as_preservation} does not hold. Thanks to assumption \eqref{eq:mean_conservation}, $(M_t)_{t\geq 0}$ is a positive martingale, thus, by Doob's martingale convergence theorem, there exists $M_\infty < \infty$ such that $M_\infty = \lim_{t\to\infty} M_t$ a.s. But, for any $0<p<1$ fixed, Fatou's lemma gives $\IE M_\infty^p \leq \lim_{t} \IE M_t^p \leq \lim_{t} e^{-\gamma t}\IE M_0^p$ with $\gamma>0$, thus $M_\infty = 0$ a.s.
\end{proof}

\begin{remark}
\label{rmk:ps_moments}
\begin{itemize}
\item When \eqref{eq:weak_as_preservation} does not hold, then the particles converge to $0$ a.s.\ when $t\to\infty$, but this convergence is degenerate in the sense that only the moments of order $p<1$ go to $0$, while the moments of order $p>1$ all diverge with $t$.

\item Theorem \ref{thm:ps_moments} establishes a qualitative difference between the kinetic equation and the finite particle system: for the kinetic equation, the moments $\int v^p f_t(dv)$ remain uniformly bounded (in time) for $p<\alpha$, where $\alpha$ is given by \eqref{eq:pareto_index} and can take any value in $(1,\infty]$, depending on the moments of $L$, $R$, $\tL$ and $\tR$; but for the particle system, the moments $\IE[(V_t^1)^p]$ blow up with $t$ for all $p>1$ as soon as \eqref{eq:weak_as_preservation} is not satisfied. This gives a negative answer to (Q1). As an extreme example: if $L$, $R$, $\tL$ and $\tR$ are independent and uniformly distributed on $[0,1]$, it is easily seen that $\alpha = \infty$, thus $\sup_{t\geq 0} \int v^p f_t(dv) < \infty$ for all $p>1$; but, since \eqref{eq:weak_as_preservation} is not satisfied, Theorem \ref{thm:ps_moments} tells us that $\sup_{t\geq 0} \IE[ (V_t^1)^p ] = \infty$ for all $p>1$.

\item
Regarding equilibration, Theorem \ref{thm:ps_moments} establishes a departure from the case of a.s.\ conservative interactions typical of the physical context: for a model with a nondegenerate equilibrium distribution $f_\infty$ and where \eqref{eq:weak_as_preservation} does not hold, we will have $\lim_t \law(V_t^1) = \delta_0$ weakly, whereas $\lim_t f_t = f_\infty \neq \delta_0$, which means that uniform propagation of chaos is impossible in this case. Thus, (Q2) and (Q3) also have negative answers in general.
\end{itemize}
\end{remark}

\begin{remark}
\label{rmk:ps_moments2}
\begin{itemize}
\item The rate $\gamma = \gamma_{N,p}$ provided by Theorem \ref{thm:ps_moments} is of order $1/N$: indeed, using the expansion $(1+x)^p \thickapprox 1 + px + p(p-1)x^2$ and condition \eqref{eq:mean_conservation}, heuristically we have
\begin{align*}
\beta
&\thickapprox \EE\left[ 1 + \frac{p}{N} [L+R+\tL+\tR-2] + \frac{p(p-1)}{N^2} [L+R+\tL+\tR-2]^2 \right] \\
&= 1 + \frac{C_p}{N^2},
\end{align*}
then $\gamma = \frac{N}{2}(1-\beta) \thickapprox \frac{-C_p}{2N}$. Thus,  when $N$ is large, if \eqref{eq:weak_as_preservation} does not hold, the convergence/divergence of the moments occurs very slowly as $t\to\infty$.

\item Although Theorem \ref{thm:ps_moments} is bad news for the particle system as an approximating tool for the kinetic equation, not all hope is lost. Firstly, as mentioned in the previous point, the degeneracy (as $t\to\infty$) of the particles is taking place very slowly for large $N$, which means that the system can still be used to efficiently approximate $f_t$ on finite time intervals. Secondly, one can work with the \emph{rescaled} particle system $\hat{\mathbf{V}}_t = (\hat{V}_t^1,\ldots,\hat{V}_t^N)$, defined as
\[
\hat{V}_t^i = \frac{V_t^i}{M_t}
\qquad \forall i=1,\ldots,N,
\]
which by definition preserves the mean wealth $\frac{1}{N} \sum_i \hat{V}_t^i$. Numerical simulations seem to indicate that this rescaled system enjoys better stability properties; in particular, its moments of order $p<\alpha$ appear to be bounded uniformly on $t$ and $N$. The extent of this and other related properties remains to be investigated mathematically.
\end{itemize}
\end{remark}

Theorem \ref{thm:ps_moments} asserts in particular that condition \eqref{eq:weak_as_preservation} is necessary to have moments $\IE[(V_t^1)^p]$ bounded for all $p>1$. Is it sufficient? For $p=2$, the answer is negative, as Proposition \ref{prop:EV2} below shows. Consider the following condition, similar to \eqref{eq:as_preservation}:
\begin{equation}
\label{eq:as_preservation_2}
L + R = 1 = \tL + \tR \quad \text{a.s.}
\end{equation}
We will need the following technical result; the proof is given in the Appendix.

\begin{lemma}
\label{lem:ABCD}
Consider the constants $\aaa = 1 - \frac{1}{2} \EE [L^2 + R^2 + \tL^2 + \tR^2]$, $\bb = \EE[ LR + \tL\tR ]$, $\cc = \EE[L\tR + \tL R]$ and $\dd = 1-\EE[L\tL + R\tR]$. Then:
\begin{enumerate}
\item[(i)] $\aaa \leq \bb$, $\aaa \leq \cc$ and $\aaa \leq \dd$,
\item[(ii)] $\aaa+\dd \leq \bb+\cc$,
\item[(iii)] if $\aaa\geq 0$, then $\aaa\dd \leq \bb\cc$,
\item[(iv)] $\aaa\dd = \bb\cc$ if and only if \eqref{eq:as_preservation} or \eqref{eq:as_preservation_2} hold.
\end{enumerate}
\end{lemma}

\begin{proposition}[boundedness of second moment]
\label{prop:EV2}
Assume $0 < \IE[(V_0^1)^2] < \infty$. Then $\sup_{t\geq 0} \IE[ (V_t^1)^2 ] < \infty$ if and only if \eqref{eq:as_preservation} or \eqref{eq:as_preservation_2} hold.
\end{proposition}

\begin{proof}
The idea is to find differential equations for the pair $g(t) := \IE[ (V_t^1)^2 ]$, and $h(t) := \IE[ V_t^1 V_t^2]$. Clearly, using exchangeability and \eqref{eq:SDEps}, for $\Phi(\mathbf{v}) = (v^1)^2$ we have
\begin{align}
\frac{dg(t)}{dt}
&= \frac{d}{dt}\IE \int_0^t \int_{\IR^4} \int_{[0,N)^2} \sum_{i\neq j}\ind_{\ii(\xi)=i,\ii(\zeta)=j} \notag \\
& \qquad \qquad \qquad {} \times [ \Phi(\mathbf{V}_{s^-}'^{ij}) - \Phi(\mathbf{V}_{s^-}) ] \mathcal{P}(ds,dl,dr,d\tl,d\tr,di,dj) \notag \\
&= \frac{1}{2(N-1)} \IE \EE \left[ \sum_{j\neq 1}\{ (L V_t^1 + R V_t^j)^2 - (V_t^1)^2 \} \right. \notag \\
& \qquad \qquad \qquad \qquad  {} \left.
+ \sum_{i\neq 1} \{ (\tL V_t^1 + \tR V_t^i)^2 - (V_t^1)^2 \} \right] \notag \\
&= \frac{1}{2} \IE \EE \left[ L^2 (V_t^1)^2 + R^2 (V_t^2)^2  + 2 LR V_t^1 V_t^2 \right. \notag \\
& \qquad \qquad \left. {} + \tL^2 (V_t^1)^2 + \tR^2 (V_t^2)^2 + 2\tL\tR V_t^1 V_t^2 - 2(V_t^1)^2  \right] \notag \\
&= -\aaa g(t) + \bb h(t), \label{eq:dgt}
\end{align}
with the notation of Lemma \ref{lem:ABCD}. Similarly, taking $\Phi(\mathbf{v}) = v^1 v^2$, for $h(t)$ we identify in the summation $\sum_{i\neq j}$ the terms where $V_t^1$ and $V_t^2$ interact directly, and the terms where either $V_t^1$ or $V_t^2$ interacts with some $V_t^i$ for $i\geq 3$. Using exchangeability, this gives:
\begin{align}
\frac{dh(t)}{dt}
&= \frac{1}{2(N-1)} \IE \EE \left[
2\{(LV_t^1 + RV_t^2)(\tL V_t^2 + \tR V_t^1) - V_t^1 V_t^2\} \right. \notag \\
& \qquad\qquad {} + 2(N-2)\{ (L V_t^1 + R V_t^3) V_t^2 - V_t^1 V_t^2 \} \notag \\
& \qquad\qquad \left. {} + 2(N-2)\{ (\tL V_t^1 + \tR V_t^3) V_t^2 - V_t^1 V_t^2 \} \right] \notag \\
&= \frac{1}{N-1} \IE \EE \left[ L\tR (V_t^1)^2 + \tL R (V_t^2)^2 + (L\tL + R\tR) V_t^1 V_t^2 - V_t^1 V_t^2 \right] \notag \\
&= \frac{1}{N-1} [ \cc g(t) - \dd h(t) ], \label{eq:dht}
\end{align}
where in the second equality we used the fact that $\EE[L+R+\tL+\tR] = 2$ to discard the last terms. Applying $(\frac{d}{dt} + \frac{\dd}{N-1})$ in \eqref{eq:dgt}, multiplying by $\bb$ in \eqref{eq:dht} and adding, we obtain
\begin{equation}
\label{eq:ddgt}
\frac{d^2g(t)}{dt^2} + \left(\aaa+\frac{\dd}{N-1}\right) \frac{dg(t)}{dt} + \frac{\aaa\dd - \bb\cc}{N-1} g(t) = 0.
\end{equation}
Call $\lambda_1$, $\lambda_2$ the roots of the corresponding characteristic polynomial, that is,
\[
\lambda_{1,2} = \frac{1}{2} \left( -\aaa-\frac{\dd}{N-1} \pm \sqrt{ \left(\aaa+\frac{\dd}{N-1}\right)^2 - 4 \frac{\aaa\dd - \bb\cc}{N-1} } \right).
\]

Now, we prove the direct implication, so we assume $\sup_{t\geq 0} \IE[ (V_t^1)^2 ] < \infty$. This implies that $\aaa\geq 0$: if not, from \eqref{eq:dgt} we would have $\frac{dg(t)}{dt} \geq -\aaa g(t)$, and then $g(t)\to\infty$. From Lemma \ref{lem:ABCD}-(iii), we thus have $\aaa\dd\leq \bb\cc$, then $\lambda_{1,2} \in \IR$. Suposse that neither \eqref{eq:as_preservation} nor \eqref{eq:as_preservation_2} hold, then we would have $\aaa\dd < \bb\cc$ by Lemma \ref{lem:ABCD}-(iv), which implies that $\lambda_2 < 0 < \lambda_1$, and the solution of \eqref{eq:ddgt} writes $g(t) = c_1 e^{\lambda_1 t} + c_2 e^{\lambda_2 t}$. Since we are assuming that $g(t)$ is bounded, necessarily $c_1 = 0$. But then $1 = (\IE[ V_t^1 ])^2 \leq \IE[ (V_t^1)^2] = c_2 e^{\lambda_2 t} \to 0$, which is a contradiction. Thus, either \eqref{eq:as_preservation} or \eqref{eq:as_preservation_2} must hold.

For the reciprocal implication, assume that either \eqref{eq:as_preservation} or \eqref{eq:as_preservation_2} holds. Then $L,R,\tL,\tR \leq 1$ a.s., thus $\dd\geq \aaa = 1- \frac{1}{2}\EE[L^2 + R^2 + \tL^2 + \tR^2] \geq 1 - \frac{1}{2}\EE[L+R+\tL+\tR] = 0$. Also, $\aaa\dd = \bb\cc$ thanks to Lemma \ref{lem:ABCD}-(iv). Thus, $\lambda_2 = -\aaa - \frac{\dd}{N-1}\leq 0$, $\lambda_1 = 0$. If $\lambda_2 < 0$, then the solution of \eqref{eq:ddgt} writes $g(t) = c_1 + c_2 e^{\lambda_2 t}$, which stays bounded. On the other hand, if $\lambda_2 = \lambda_1 = 0$, then $g(t) = g(0) + \frac{dg(0)}{dt} $. But $\lambda_2 = \lambda_1 = 0$ means that $\aaa=\dd=0$, which gives $\bb=0$ also (because \eqref{eq:as_preservation} implies $\bb=\dd$, whereas \eqref{eq:as_preservation_2} implies $\bb=\aaa$). Since $\aaa=\bb=0$, form \eqref{eq:dgt} we obtain $\frac{dg(0)}{dt} = 0$, thus $g(t) \equiv g(0)$, which also stays bounded.
\end{proof}

\begin{remark}
\label{rmk:EV2}
\begin{itemize}
\item This result shows that, in general, condition \eqref{eq:weak_as_preservation} is not sufficient to ensure the boundedness of the moments of the particle system. For instance: if $L,R$ are independent and uniformly distributed on $[0,1]$, and $\tL = 1-L$, $\tR = 1-R$, then \eqref{eq:weak_as_preservation} is satisfied, but still $\sup_t \IE[ (V_t^1)^2 ] = \infty$ because neither \eqref{eq:as_preservation} nor \eqref{eq:as_preservation_2} holds.

\item Condition \eqref{eq:as_preservation_2} means that the post-trade riches $v'$ and $v_*'$ are (random) linear combinations of $v$ and $v_*$, thus $v',v_*' \in [\min(v,v_*),\max(v,v_*)]$. This produces ``agglomeration'' of the system, in which particles become closer and closer together with time.

\item Numerical simulations seem to indicate that the degenerate behaviour of the moments of the particles described in Theorem \ref{thm:ps_moments} occurs as soon as both \eqref{eq:as_preservation} and \eqref{eq:as_preservation_2} do not hold, even if \eqref{eq:weak_as_preservation} does. We thus believe that the conclusion of Theorem \ref{thm:ps_moments} is still valid in this case; more specifically, we conjecture that
\[
\text{\eqref{eq:as_preservation} and \eqref{eq:as_preservation_2} do not hold}
\Leftrightarrow
\begin{cases}
\lim_{t\to \infty} \IE[ (V_t^1)^p ] = 0 \quad\forall 0<p<1, \\
\lim_{t\to \infty} \IE[ (V_t^1)^p ] = \infty \quad \forall p>1, \\
\lim_{t\to \infty} V_t^1 = 0 \text{ a.s.}
\end{cases}
\]
\end{itemize}
\end{remark}

\section{Strictly conservative economies}
\label{sec:strictly-conserv}

The results of the previous section imply that, unless one assumes some kind of a.s.\ preservation condition on the interaction coefficients (like \eqref{eq:as_preservation} or \eqref{eq:as_preservation_2}), there is no hope for nice stability properties of the particle system, such as moments propagation or uniform propagation of chaos. We now investigate these properties when one does assume such a condition; more specifically, we will assume throughout this section that the interactions are \emph{strictly conservative}, i.e., they satisfy $L + \tR = 1 = \tL + R$ a.s.\ (condition \eqref{eq:as_preservation}). 

\begin{remark}
One could also consider the case where the interactions satisfy $L+R=1=\tL+\tR$ a.s.\ (condition \eqref{eq:as_preservation_2}). However, the long time behaviour of $f_t$ is somewhat trivial in this case: as shown in \cite[Theorem 4.3]{matthes-toscani2008}, the equilibrium distribution $f_\infty$ is a Dirac mass at $m = \int v f_0(dv)$. Given the ``agglomeration'' phenomenon mentioned in Remark \ref{rmk:EV2}, a similar behaviour es expected for the particle system. For this section, we thus decided to focus on condition \eqref{eq:as_preservation} (which is more meaningful from the economic point of view) and leave out the case \eqref{eq:as_preservation_2}.
\end{remark}

We will quantify convergence of distributions with the \emph{$2$-Wasserstein distance}: for probability measures $\mu, \nu$ on $\IR^k$ with finite second moment, it is defined as
\[
W_2(\mu,\nu)
= \left( \inf_{\mathbf{X},\mathbf{Y}} \IE \left[ \frac{1}{k} \sum_{i=1}^k (X^i - Y^i)^2 \right] \right)^{1/2},
\]
where the infimum is taken over all possible couplings of $\mu$ and $\nu$, i.e., over all random vectors $\mathbf{X} = (X^1,\ldots,X^k)$ and $\mathbf{Y} = (Y^1,\ldots,Y^k)$ such that $\law(\mathbf{X}) = \mu$ and $\law(\mathbf{Y}) = \nu$. The factor $\frac{1}{k}$ in front of the summation is natural when one cares about the dependence on the dimension. One of the advantages of the Wasserstein distance is that it is relatively easy to bound from above: given \emph{any} coupling $(\mathbf{X},\mathbf{Y})$, the quantity $\IE \frac{1}{k} \sum_i (X^i-Y^i)^2$ provides an upper bound for $W_2^2(\mu,\nu)$; this is the overall strategy we use to prove the upcoming results. It can be shown that the infimum is always achieved by some $(\mathbf{X},\mathbf{Y})$, and such a pair is called an \emph{optimal coupling}, see \cite{villani2009} for more information on couplings and Wasserstein distances. 

We now state and prove our results.

\begin{theorem}[contractivity, strictly conservative case]
\label{thm:contraction}
Assume \eqref{eq:as_preservation}. Let $\mathbf{V}_t$ and $\mathbf{U}_t$ be two solutions to \eqref{eq:SDEps}, using the same Poisson point measure $\mathcal{P}$, and starting from (possibly distinct) exchangeable initial conditions $\mathbf{V}_0$ and $\mathbf{U}_0$ having the same total initial wealth, i.e., $\sum_i V_0^i = \sum_i U_0^i$ a.s., thus $\sum_i V_t^i = \sum_i U_t^i$ for all $t \geq 0$ a.s. Then, for $\aaa = 1 - \frac{1}{2} \EE [L^2 + R^2 + \tL^2 + \tR^2]$ and $\bb = \EE[ LR + \tL\tR ]$, we have for all $t\geq 0$
\[
W_2^2(\law(\mathbf{V}_t), \law(\mathbf{U}_t))
\leq \IE[ (V_t^1 - U_t^1)^2 ]
= e^{-(\aaa + \frac{\bb}{N-1})t} \IE[ (V_0^1 - U_0^1)^2 ].
\]
\end{theorem}

\begin{proof}
Define $g(t) = \IE[(V_t^1 - U_t^1)^2 ] \geq W_2^2(\law(\mathbf{V}_t), \law(\mathbf{U}_t))$ (by exchangeability) and $h(t) = \IE[(V_t^1 - U_t^1)(V_t^2 - U_t^2)]$. From the SDE \eqref{eq:SDEps}, a similar computation as in the proof of Proposition \eqref{prop:EV2} shows that $g(t)$ satisfies the same differential equation \eqref{eq:dgt}, i.e., $\frac{dg(t)}{dt} = -\aaa g(t) + \bb h(t)$. But, using exchangeability and the fact that $\sum_{i=2}^N (V_t^i - U_t^i) = -(V_t^1 - U_t^1)$, we have
\[
h(t)
= \IE\left[ (V_t^1-U_t^1) \frac{1}{N-1} \sum_{i=2}^N (V_t^i - U_t^i) \right]
= \frac{-g(t)}{N-1},
\]
thus $\frac{dg(t)}{dt} = -(\aaa + \frac{\bb}{N-1}) g(t)$, which proves the claim.
\end{proof}

\begin{remark}
\label{rmk:contraction}
\begin{itemize}
\item Under condition \eqref{eq:as_preservation} we obviously have $L,R,\tL,\tR \leq 1$ a.s., thus $\aaa \geq 0$; moreover, $\aaa = 0$ is equivalent to
\begin{equation}
\label{eq:LR01}
L,R,\tL,\tR \in\{0,1\} \quad \text{a.s.}
\end{equation}
One particular example of this is the ``winner takes all'' dynamics considered in \cite{matthes-toscani2008}: $v' = v+v_*$ and $v_*' = 0$, i.e., one agent loses all her money to the other one. In terms of the interaction coefficients, this corresponds to $L = R = 1$ and $\tL = \tR = 0$ a.s., which of course yields $\aaa = 0$, but also $\bb = \EE[LR + \tL\tR] = 1$; thus, Theorem \ref{thm:contraction} gives contraction at slow exponential rate $\frac{1}{N-1}$.

\item The only case where Theorem \ref{thm:contraction} does not give contraction is when $\aaa = 0$ and $\bb = 0$, which means that $L,R,\tL,\tR \in\{0,1\}$ and $LR = 0 = \tL\tR$ a.s. But this implies that $L+R=1=\tL+\tR$ a.s., then either $v' = v$ and $v_*' = v_*$ (no trade), or $v' = v_*$ and $v_*' = v$ (full exchange). Thus, no effective trading is taking place between the agents, in the sense that the empirical distribution $\frac{1}{N} \sum_i \delta_{V_t^i}$ remains constant a.s. This is consistent with the lack of contraction.

\item Theorem \ref{thm:contraction}, Corollary \ref{cor:equilibration} below, and the two previous points, are to be compared with \cite[Theorem 4.1]{matthes-toscani2008}, which proves an analogous behaviour for the flow $(f_t)_{t\geq 0}$.
\end{itemize}
\end{remark}

As an immediate consequence, we obtain the following result, which gives a positive answer to (Q2) under \eqref{eq:as_preservation}:

\begin{corollary}[equilibration, strictly conservative case]
\label{cor:equilibration}
Assume \eqref{eq:as_preservation} and that $\aaa+\bb>0$. Then, for each $m>0$, there exists a unique measure $\mu_\infty$ on the simplex $S_m = \{\mathbf{v}\in\IR_+^N: \frac{1}{N} \sum_i v^i = m \}$ of mean wealth $m$, such that for every exchangeable initial condition satisfying $\mathbf{V}_0 \in S_m$ a.s., we have
\[
W_2^2(\law(\mathbf{V}_t), \mu_\infty)
\leq e^{-(\aaa + \frac{\bb}{N-1})t} W_2^2(\law(\mathbf{V}_0), \mu_\infty)
\qquad \forall t\geq 0.
\]
\end{corollary}

\begin{proof}
Let $\mathcal{P}(S_m)$ be the space of probability measures on $S_m$, endowed with the topology of weak convergence. From \eqref{eq:SDEps}, we see that the flow $\mu_t = \law(\mathbf{V}_t) \in \mathcal{P}(S_m)$ solves
\begin{equation}
\label{eq:Amu}
\frac{d\mu_t}{dt} = \frac{N}{2}(A - \text{Id}) \mu_t,
\end{equation}
where $A: \mathcal{P}(S_m) \to \mathcal{P}(S_m)$ is the operator given by
\[
\langle A\mu, \Phi \rangle
= \frac{1}{N(N-1)} \sum_{i\neq j} \int_{S_m} \EE \Phi(\mathbf{v}'^{ij}) \mu(d\mathbf{v}),
\]
for every measurable and bounded function $\Phi:S_m \to \IR$. Since $\mathcal{P}(S_m)$ is compact and $A$ is continuous, there exists $\mu_\infty$ such that $A\mu_\infty = \mu_\infty$, which implies that $\mu_t \equiv \mu_\infty$ is a stationary solution of \eqref{eq:Amu}. Taking $(\mathbf{V}_0, \mathbf{U}_0)$ as an optimal coupling between $\law(\mathbf{V}_0)$ and $\mu_\infty$ in Theorem \ref{thm:contraction}, yields the desired estimate, because $\law(\mathbf{U}_t) = \mu_\infty$, $\forall t\geq 0$. Uniqueness of $\mu_\infty$ is immediate.
\end{proof}

\begin{remark}
In contrast with Kac's particle system, whose unique equilibrium is the uniform distribution on the sphere $\{\mathbf{v}\in\IR^N: \frac{1}{N}\sum_i (v^i)^2 = 1\}$ of unit mean energy, the equilibrium distribution $\mu_\infty$ provided by the previous corollary is not explicit in general. One particular case where $\mu_\infty$ is explicit is the ``winner takes all'' dynamics mentioned in Remark \ref{rmk:contraction}: it is easily seen that $\mu_\infty$ is the uniform distribution on the set of points of the form $(0,\ldots,m,\ldots,0)$ (i.e., the extreme points of $S_m$).
\end{remark}

For the following results, we will need to discard the degenerate behaviour mentioned in Remark \ref{rmk:contraction}, for which we will assume that \eqref{eq:LR01} does not hold, i.e.,
\begin{equation}
\label{eq:LRnot01}
\PP( L,R,\tL,\tR \in\{0,1\} ) < 1,
\end{equation}
which, together with \eqref{eq:as_preservation}, implies that $\aaa > 0$. The next proposition provides propagation of moments uniformly in $t$ and $N$, answering (Q1) affirmatively in the case of strictly conservative economies. For simplicity we assume fixed mean initial wealth, but it can be easily generalized to any exchangeable initial condition (as with the previous corollary).

\begin{proposition}[propagation of moments, strictly conservative case]
\label{prop:moments_prop}
Assume \eqref{eq:as_preservation}, \eqref{eq:LRnot01}, and that $\frac{1}{N} \sum_i V_0^i = m := \int v f_0(dv)$ a.s., thus $\frac{1}{N} \sum_i V_t^i = m$ a.s.\ for all $t$. Then, for all $p>1$ there exists a constant $C_p<\infty$ depending only on $p$ and the $p$-moments of $L$, $R$, $\tL$ and $\tR$, such that
\[
\IE[ (V_t^1)^p ] \leq \IE[ (V_0^1)^p ] + C_p m^p  \qquad \forall t\geq 0.
\]
\end{proposition}

\begin{proof}
The argument is similar to the one used to deduce propagation of moments for the particle system with fixed initial energy in the Boltzmann case, see for instance \cite[Lemma 5.3]{mischler-mouhot2013} or \cite[Corollary 17]{cortez-fontbona2018}. For $p>1$ fixed, call $g(t) = \IE[(V_t^1)^p]$.  Arguing as in the proof of Proposition \ref{prop:EV2}, using exchangeability and the inequality $(x+y)^p \leq x^p + y^p + 2^{p-1}(xy^{p-1} + x^{p-1}y)$, we see that it satisfies
\begin{align}
\frac{dg(t)}{dt}
&= \frac{1}{2(N-1)} \IE \EE \left[ \sum_{j\neq 1}\{ (L V_t^1 + R V_t^j)^p - (V_t^1)^p \} \right. \notag \\
& \qquad\qquad\qquad\qquad \left. {} + \sum_{i\neq 1} \{ (\tL V_t^1 + \tR V_t^i)^p - (V_t^1)^p \} \right] \notag \\
&= -g(t) + \frac{1}{2} \IE \EE [ (L V_t^1 + R V_t^2)^p  + (\tL V_t^1 + \tR V_t^2)^p ] \notag \\
&\leq -\aaa_p g(t) + \bb_p \IE[(V_t^1)^{p-1} V_t^2], \label{eq:dgt2}
\end{align}
where $\aaa_p := 1 - \frac{1}{2} \EE[L^p + R^p + \tL^p + \tR^p] > 0$ and $\bb_p := 2^{p-2} \EE[L^{p-1}R + LR^{p-1} + \tL^{p-1}\tR + \tL\tR^{p-1}]$. Now, exchangeability and the fact that $\sum_{i=2}^N V_t^i \leq Nm$ a.s.\ gives us
\[
\IE[(V_t^1)^{p-1} V_t^2]
= \IE\left[(V_t^1)^{p-1} \frac{1}{N-1} \sum_{i=2}^N V_t^i \right]
\leq \frac{Nm \IE[(V_t^1)^{p-1}]}{N-1}
\leq 2m g(t)^{1-1/p}.
\]
Thus, from \eqref{eq:dgt2} we obtain $\frac{dg(t)}{dt} \leq -\aaa_p g(t) + 2m\bb_p g(t)^{1-1/p}$. This differential inequality implies that $g(t) \leq \max(g(0), x^*) \leq g(0) + x^*$ for all $t$, where $x^* = (2m\bb_p /\aaa_p)^p$ is the unique positive root of the polynomial $-\aaa_p x + 2m \bb_p x^{1-1/p}$. This proves the desired bound.
\end{proof}

Recall that the kinetic equation \eqref{eq:kinetic} propagates the moments of order $p<\alpha$, where $\alpha$ is the Pareto index of $f_\infty$ given by \eqref{eq:pareto_index}, see for instance \cite[Theorem 3.2]{matthes-toscani2008} or \cite[Lemma 5]{cortez-fontbona2016}. Assuming \eqref{eq:as_preservation} and the non-degeneracy condition \eqref{eq:LRnot01}, we have $\alpha = \infty$, thus
\begin{equation}
\label{eq:moments_ft}
\forall p\geq 0,
\quad \int v^p f_0(dv) < \infty
\quad \Rightarrow \quad
\sup_{t\geq 0} \int v^p f_t(dv) < \infty.
\end{equation}
The next theorem, whose proof is given at the end of this section, provides a uniform (in time) propagation of chaos rate for the particle system. It is stated in terms of its \emph{empirical measure}: given $\mathbf{v} = (v^1,\ldots,v^N) \in \IR^N$, we denote
\[
\bar{\mathbf{v}}
= \frac{1}{N} \sum_{j=1}^N \delta_{v^j}
\qquad \text{and} \qquad
\bar{\mathbf{v}}^i
= \frac{1}{N-1} \sum_{j\neq i} \delta_{v^j} \quad \text{for $i=1,\ldots,N$.}
\]

\begin{theorem}[uniform propagation of chaos, strictly conservative case] \label{thm:UPoC}
Assume \eqref{eq:as_preservation}, \eqref{eq:LRnot01}, and that $\frac{1}{N} \sum_i V_0^i = m := \int v f_0(dv)$ a.s., thus $\frac{1}{N} \sum_i V_t^i = m$ a.s.\ for all $t$. Assume also that $\int v^q f_0(dv) < \infty$ for some $q>4$. Let $\aaa = 1- \frac{1}{2}\EE[L^2+R^2+\tL^2+\tR^2]>0$. Then there exists $C>0$ only depending on $\aaa$, $\IE[(V_0^1)^2]$, and the uniform bound of $\int v^q f_t(dv)$ provided by \eqref{eq:moments_ft}, such that
\[
\IE W_2^2( \bar{\mathbf{V}}_t, f_t )
\leq 4e^{-\aaa t} W_2^2(\law(\mathbf{V}_0),f_0^{\otimes N}) + \frac{C}{N^{1/3}} \qquad \forall t \geq 0.
\]
\end{theorem}

\begin{remark}
\begin{itemize}
\item Thus, this theorem gives a  chaos at rate of order $N^{-1/3}$, provided that $W_2^2(\law(\mathbf{V}_0),f_0^{\otimes N})$ converges to 0 at that rate or faster, answering (Q3) affirmatively. This is quite reasonable, considering that $N^{-1/2}$ is the best general rate of convergence for the empirical measure of an i.i.d.\ sequence towards its common law, see \cite[Theorem 1]{fournier-guillin2013}. Also, $N^{-1/3}$ is the same chaos rate obtained previously for more physical models, see for instance \cite{cortez-fontbona2018} for the Boltzmann equation and \cite{cortez2016} for Kac's model. Regarding time dependence, up to our knowledge, Theorem \ref{thm:UPoC} is the first uniform propagation of chaos result in the context of wealth redistribution models; the only related result, found in \cite{cortez-fontbona2016}, provides a non-uniform chaos rate in 1-Wasserstein distance that grows linearly with time, for the general case of interactions preserving wealth only in the mean (this is of course expected, in light of the findings of Section \ref{sec:mean-preserving}).

\item In Theorem \ref{thm:UPoC}, the hypothesis $\int v^q f_0(dv) < \infty$ for some $q>4$ can be relaxed to only $2<q<4$, obtaining a slower chaos rate of order $N^{-\eta}$ for $\eta = \frac{q-2}{2q-2}<1/3$. This is a consequence of using \cite[Theorem 1]{fournier-guillin2013} in the proof, see \eqref{eq:epsk} below.

\item Theorem \ref{thm:UPoC} gives a mathematical justification to the observation of ``absence of finite-size effects'' in the particle system, made by Chakraborti and Chakrabarti in \cite{chakraborti-chakrabarti2000} based on numerical simulations in the case of strictly conservative economies.
\end{itemize}
\end{remark}

Naturally, equilibration for the particle system together with uniform propagation of chaos allow to easily deduce convergence of the equilibrium distribution:

\begin{corollary}[convergence of equilibrium distribution, strictly conservative case]
\label{cor:conv_equilibrium}
Assume the same hypotheses as in Theorem \ref{thm:UPoC}. Let $\mathbf{V}_\infty$ be a random vector on $\IR^N$ with $\law(\mathbf{V}_\infty) = \mu_\infty$, where $\mu_\infty$ is the equilibrium distribution of the particle system given by Corollary \ref{cor:equilibration}. Then there exists $C>0$ depending on the same quantities as in Theorem \ref{thm:UPoC}, such that 
\[
\IE W_2^2(\bar{\mathbf{V}}_\infty, f_\infty)
\leq \frac{C}{N^{1/3}}
\qquad \forall t\geq 0.
\]
\end{corollary}

\begin{proof}
Let $(\mathbf{V}_t)_{t\geq 0}$ be the particle system starting with $\law(\mathbf{V}_0) = \mu_\infty$. Therefore, as seen in the proof of Corollary \ref{cor:equilibration}, we have $\law(\mathbf{V}_t) = \mu_\infty$ for all $t\geq 0$. Then
\[
\IE W_2^2(\bar{\mathbf{V}}_\infty, f_\infty)
= \IE W_2^2(\bar{\mathbf{V}}_t, f_\infty)
\leq 2\IE W_2^2(\bar{\mathbf{V}}_t, f_t)
     +2W_2^2(f_t, f_\infty),
\]
for all $t \geq 0$. From \cite[Theorem 5]{bassetti-ladelli-matthes2011}, we know that $W_2(f_t,f_\infty) \to 0$ as $t\to\infty$. Using this, Theorem \ref{thm:UPoC}, and letting $t\to\infty$ in the last inequality, yields the result.
\end{proof}

To prove Theorem \ref{thm:UPoC}, we will make use of a coupling argument introduced in \cite{cortez-fontbona2016} and later used in \cite{cortez2016} and \cite{cortez-fontbona2018}. The main idea is to couple the particle system $\mathbf{V}_t = (V_t^1,\ldots,V_t^N)$ with a system $\mathbf{Z}_t = (Z_t^1,\ldots,Z_t^N)$, where each $Z_t^i$ is a \emph{nonlinear process} (defined below) that remains close to $V_t^i$. To proceed with this coupling construction, from \eqref{eq:SDEps}, we first notice that for each $i = 1,\ldots,N$, the $i$-th particle of the system satisfies the SDE
\begin{equation}
\label{eq:dVi}
dV_t^i
= \int_{\IR^2} \int_{[0,N)} [lV_{t^-}^i + rV_{t^-}^{\ii(\xi)} - V_{t^-}^i] \mathcal{P}^i(dt,dl,dr,d\xi),
\end{equation}
where $\mathcal{P}^i$ is the Poisson point measure on $[0,\infty) \times \IR^2 \times [0,N)$ given by
\begin{equation}
\label{eq:Pi}
\mathcal{P}^i(dt,dl,dr,d\xi)
= \mathcal{P}(dt,dl,dr,\IR,\IR, [i-1,i), d\xi) 
 + \mathcal{P}(dt,\IR,\IR,dl, dr, d\xi, [i-1,i)),
\end{equation}
that is, $\mathcal{P}^i$ selects the atoms of $\mathcal{P}$ that induce a jump on particle $V_t^i$, i.e., those where either $\ii(\xi) = i$ or $\ii(\zeta) = i$. Notice that $\mathcal{P}^i$ is a Poisson point measure with intensity
\[
dt \bar{\Lambda}(dl,dr) \frac{d\xi \ind_{A^i}(\xi)}{N-1},
\]
where $\bar{\Lambda} := \frac{1}{2} \law(L,R) + \frac{1}{2} \law(\tL,\tR)$ and $A^i := [0,N) \setminus [i-1,i)$.

The \emph{nonlinear process} (introduced by Tanaka in \cite{tanaka1979} in the context of the Boltzmann equation) is the probabilistic counterpart of the kinetic equation \eqref{eq:kinetic}, and it represents the trajectory of a single agent inmersed an infinite population. It is a stochastic pure-jump process having marginal laws $(f_t)_{t\geq 0}$, and it can be defined for instance as the solution to \eqref{eq:dVi} where $V_{t^-}^{\ii(\xi)}$, which is a $\xi$-realization of the (random) measure $\bar{\mathbf{V}}_{t^-}^i = \frac{1}{N-1} \sum_{j\neq i} \delta_{V_{t^-}^j}$, is replaced with a realization of the measure $f_t$.

The key idea, introduced in \cite{cortez-fontbona2016}, is to define, for each $i=1,\ldots,N$, a nonlinear process $Z_t^i$ that mimics as closely as possible the dynamics \eqref{eq:dVi} of particle $V_t^i$. More specifically, $Z_t^i$ is defined as the unique jump-by-jump solution of the SDE
\begin{equation}
\label{eq:dZi}
dZ_t^i
= \int_{\IR^2} \int_{[0,N)} [lZ_{t^-}^i + r F_t^i(\mathbf{Z}_{t^-},\xi) - Z_{t^-}^i] \mathcal{P}^i(dt,dl,dr,d\xi).
\end{equation}
Here $F^i$ is a measurable mapping $[0,\infty) \times \IR^N \times A^i \ni (t,\mathbf{z},\xi) \to F_t^i(\mathbf{z},\xi) \in \IR$ with the property that, for each $t\geq 0$, $\mathbf{z} \in \IR^N$ and any random variable $\xi$ uniformly distributed on $A^i$, the pair $(z^{\ii(\xi)}, F_t^i(\mathbf{z},\xi))$ is an optimal coupling between $\bar{\mathbf{z}}^i = \frac{1}{N-1} \sum_{j\neq i} \delta_{z^j}$ and $f_t$, thus
\begin{equation}
\label{eq:intzFi2}
\int_{A^i} (z^{\ii(\xi)} - F_t^i(\mathbf{z},\xi))^2 \frac{d\xi}{N-1}
= W_2^2(\bar{\mathbf{z}}^i,f_t).
\end{equation}
See \cite[Lemma 3]{cortez-fontbona2016} for a proof of existence of such a mapping. The initial conditions $Z_0^1,\ldots,Z_0^N$ are chosen independently with common law $f_0$, in such a way that the pair $(\mathbf{V}_0,\mathbf{Z}_0)$ is an optimal coupling between $\law(\mathbf{V}_0)$ and $f_0^{\otimes N}$.

By construction, each $Z_t^i$ is a nonlinear process, thus $\law(Z_t^i) = f_t$ for all $t\geq 0$. Notice however that $Z_t^i$ and $Z_t^j$ have a simultaneous jump whenever $V_t^i$ and $V_t^j$ interact, which implies that $Z_t^1,\ldots,Z_t^N$ are \emph{not independent}. To use this construction, we will need to prove that these nonlinear processes become \emph{assymptotically independent uniformly on time} as $N\to\infty$, which is stated in the next lemma. It is almost the same as \cite[Lemma 6]{cortez-fontbona2016} or \cite[Lemma 4]{cortez2016} with minor differences; for convenience of the reader, we sketch the proof in the Appendix.

\begin{lemma}[decoupling of the nonlinear processes]
\label{lem:decoupling}
Assume conditions \eqref{eq:as_preservation}, \eqref{eq:LRnot01}, and that $\int v^q f_0(dv) < \infty$ for some $q>4$. Then there exists a constant $C>0$ only depending on $\aaa>0$ and the uniform bound of $\int v^q f_t(dv)$ provided by \eqref{eq:moments_ft}, such that
\[
\IE W_2^2( \bar{\mathbf{Z}}_t , f_t) \leq \frac{C}{N^{1/3}}
\qquad \forall t\geq 0.
\]
Moreover, the same bound holds for $\IE W_2^2( \bar{\mathbf{Z}}_t^1 , f_t)$.
\end{lemma}

We will also need the following result, which is very similar to \cite[Lemma 4.1.8]{ambrosio-gigli-savare2008}. We also give a proof in the Appendix.

\begin{lemma}[a version of Gr\"onwall's lemma]
\label{lem:Gronwall}
Let $u:\IR_+ \to \IR_+$ be a non-negative function satisfying $\frac{du}{dt} \leq -au + bu^{1/2} + c$ for some constants $a>0$, $b\geq 0$ and $c\geq 0$. Then,
\[
u(t)
\leq 2u(0) e^{-at} +\frac{2c}{a} + \frac{4b^2}{a^2}
\quad \forall t \geq 0.
\]
\end{lemma}

We are now ready to prove Theorem \ref{thm:UPoC}.

\begin{proof}[Proof of Theorem \ref{thm:UPoC}]
Call $g(t) = \IE[ (V_t^1 - Z_t^1)^2 ]$ and $h(t) = \IE W_2^2(\bar{\mathbf{Z}}_t^1,f_t)$. Let us shorten notation: call $V = V_t^1$, $V_* = V_t^{\ii(\xi)}$, $Z = Z_t^1$, $F = F_t^1(\mathbf{Z}_t,\xi)$ and $Z_* = Z_t^{\ii(\xi)}$ (or with $s^-$ in place of $t$). From \eqref{eq:dVi} and \eqref{eq:dZi}, we have
\begin{align}
\frac{dg(t)}{dt}
&= \frac{d}{dt} \IE \int_0^t \int_{\IR^2} \int_{A^1} \left[ (lV + rV_* -lZ - rF)^2  - (V - Z)^2 \right] \mathcal{P}^1(ds,dl,dr,d\xi) \notag \\
&= \IE \int_{\IR^2} \int_{A^1} \left[ (l^2-1)(V-Z)^2 + r^2 (V_* - Z_*)^2 + r^2(Z_*-F)^2 \right. \notag \\
& \qquad \qquad \left. {} + 2 r^2 (V_*-Z_*)(Z_*-F) + 2lr (V-Z) (V_* - F) \vphantom{(V)^2} \right] \frac{\bar{\Lambda}(dl,dr) d\xi}{N-1} \notag \\
&= - \aaa g(t) + K h(t) \notag \\
& \qquad {}+ \IE \int_{A^1} \left[2 K (V_*-Z_*)(Z_*-F) + \bb (V-Z) (V_* - F) \right] \frac{d\xi}{N-1}, \label{eq:dgt3}
\end{align}
where $\aaa = 1 - \frac{1}{2}\IE[L^2+R^2+\tL^2+\tR^2]>0$, $\bb = \IE[LR+\tL\tR]\leq 2$ and $K = \frac{1}{2}\EE[R^2+\tR^2]\leq 1$. Here we have used that $\IE\int_{A^1} (V_*-Z_*)^2 \frac{d\xi}{N-1} = g(t)$ thanks to exchangeability, and $\IE\int_{A^1} (Z_*-F)^2 \frac{d\xi}{N-1} = h(t)$ thanks to \eqref{eq:intzFi2}. Using that $\int_{A^1} V_* d\xi = \sum_{i=2}^N V_t^i = Nm - V_t^1$ and that $\int_{A^1} F \frac{d\xi}{N-1} = \int v f_t(dv) = m$, for the second term in the integral in \eqref{eq:dgt3} we have
\[
\IE \int_{A^1} (V-Z) (V_* - F) \frac{d\xi}{N-1}
= \frac{1}{N-1} \IE [(V-Z)(m - V_t^1)]
\leq \frac{C}{N},
\]
where we have used the fact that both $V_t^1$ and $Z_t^1$ have uniformly bounded second moment thanks to Proposition \ref{prop:moments_prop} and \eqref{eq:moments_ft}, and $C>0$ is a constant that depends on the asserted quantities, and may change from line to line. For the first term on the integral in \eqref{eq:dgt3}, using the Cauchy-Schwartz inequality and exchangeability we simply have $\IE \int_{A^1} (V_*-Z_*) (Z_* - F) \frac{d\xi}{N-1} \leq g(t)^{1/2} h(t)^{1/2}$. With all this, from \eqref{eq:dgt3} we obtain
\begin{align*}
\frac{dg(t)}{dt}
&\leq -\aaa g(t) + 2K g(t)^{1/2} h(t)^{1/2} + K h(t) + \frac{C}{N} \\
&\leq -\aaa g(t) + 2 g(t)^{1/2} (C N^{-1/3})^{1/2} +  CN^{-1/3},
\end{align*}
where we have used Lemma \ref{lem:decoupling} to bound $h(t) \leq CN^{-1/3}$ uniformly on $t$. Using Lemma \ref{lem:Gronwall}, this differential inequality implies that $g(t) \leq 2g(0) e^{-\aaa t} + CN^{-1/3}$, where $g(0) = \IE \frac{1}{N} \sum_i (V_0^i - Z_0^i)^2 = W_2^2(\law(\mathbf{V}_0),f_0^{\otimes N})$ because $(\mathbf{V}_0, \mathbf{Z}_0)$ is an optimal coupling. Notice also that $\IE W_2^2(\bar{\mathbf{V}}_t , \bar{\mathbf{Z}}_t) \leq \IE \frac{1}{N} \sum_i (V_t^i - Z_t^i)^2 = g(t)$. Finally, using all this and Lemma \ref{lem:decoupling} again, we obtain the desired estimate:
\begin{align*}
\IE W_2^2(\bar{\mathbf{V}}_t , f_t)
&\leq 2 \IE W_2^2(\bar{\mathbf{V}}_t , \bar{\mathbf{Z}}_t)
+ 2 \IE W_2^2(\bar{\mathbf{Z}}_t , f_t) \\
&\leq 4e^{-\aaa t} W_2^2(\law(\mathbf{V}_0),f_0^{\otimes N}) + \frac{C}{N^{1/3}}.
\qedhere
\end{align*}
\end{proof}

\section*{Appendix}

\begin{proof}[Proof of Lemma \ref{lem:ABCD}.]
We first prove (i): $\aaa\leq \dd$ is trivial, while for $\bb$ we have
\begin{align}
2\bb - 2\aaa
&= \EE[(L + R)^2 + (\tL + \tR)^2] - 2 \notag \\
&\geq (\EE[L+R])^2 + (\EE[\tL+\tR])^2 -2 \label{eq:AB} \\
&\geq 2 \left( \frac{1}{2} \EE[L+R+\tL+\tR] \right)^2 - 2
= 0, \notag
\end{align}
thanks to condition \eqref{eq:mean_conservation}. This proves $\aaa\leq \bb$, and $\aaa\leq \cc$ is analogous. Now we prove (ii): set $\hat{L} = 1-L-\tR$ and $\hat{R} = 1-\tL-R$, so $\EE\hat{L} = \EE\hat{R} = 0$ thanks to \eqref{eq:mean_conservation}. It is straightforward to verify that $\cc = \aaa + \frac{1}{2}\EE(\hat{L}^2 + \hat{R}^2)$ and $\bb = \dd + \EE[\hat{L}\hat{R}]$, thus $\bb+\cc =\dd+\aaa+\frac{1}{2}\EE[(\hat{L} + \hat{R})^2] \geq \dd+\aaa$. Now (iii): call $\epsilon = \cc-\aaa\geq 0$ and $\eta = \bb-\aaa \geq 0$, thus $\bb\cc = \aaa^2 + \aaa(\epsilon+\eta) + \epsilon\eta$. But $\dd \leq \aaa+\epsilon+\eta$ thanks to (ii), and assuming $\aaa\geq 0$, this gives $\aaa\dd \leq \aaa^2 + \aaa(\epsilon + \eta) \leq \bb\cc$. Finally, we prove (iv): if $\aaa\dd = \bb\cc$, then $\epsilon\eta = 0$, thus either $\epsilon = 0$ or $\eta = 0$. In the first case we have $\aaa=\bb$, thus the inequality \eqref{eq:AB} collapses, which gives $\text{var}(L+R) = \text{var}(\tL+\tR) = 0$; this means that \eqref{eq:as_preservation_2} holds. Similarly, when $\eta = 0$, we deduce that \eqref{eq:as_preservation} holds. This proves the direct implication in (iv). Reciprocally: when \eqref{eq:as_preservation} holds, we can write $R = 1-\tL$ and $\tR = 1-L$ a.s., and then a straigthforward computation shows that $\aaa=\cc$ and $\bb=\dd$; if \eqref{eq:as_preservation_2} holds, then $R = 1-L$ and $\tR = 1-\tL$ a.s., which gives $\aaa=\bb$ and $\cc=\dd$. In either case, we obtain $\aaa\dd = \bb\cc$.
\end{proof}

\begin{proof}[Proof of Lemma \ref{lem:decoupling}.]
We follow the proof of \cite[Lemma 6]{cortez-fontbona2016} and \cite[Lemma 4]{cortez2016}. We will first show that for all $k\in\{2,\ldots,N\}$, we have
\begin{equation}
\label{eq:W2Zkftk}
W_2^2(\law(Z_t^1,\ldots,Z_t^k), f_t^{\otimes k})
\leq \frac{Ck}{N}
\qquad \forall t\geq 0.
\end{equation}
To this end, we will again use a coupling argument. For each such $k$, the idea is to define independent nonlinear processes $\tilde{Z}_t^1,\ldots,\tilde{Z}_t^k$ that remain close to $Z_t^1,\ldots,Z_t^k$ on expectation. Consider $\tilde{\mathcal{P}}$ an independent copy of the Poisson point measure $\mathcal{P}$, and for each $i=1,\ldots,k$, define
\begin{align*}
\tilde{\mathcal{P}}^i(dt,dl,dr,d\xi)
&= \mathcal{P}(dt,dl,dr,\IR,\IR, [i-1,i), d\xi) \\
& \qquad {} + \mathcal{P}(dt,\IR,\IR,dl, dr, d\xi, [i-1,i)) \ind_{[k,N)}(\xi) \\
& \qquad {} + \tilde{\mathcal{P}}(dt,\IR,\IR,dl, dr, d\xi, [i-1,i)) \ind_{[0,k)}(\xi),
\end{align*}
which is a Poisson point measure with intensity $dt \bar{\Lambda}(dl,dr) \frac{d\xi \ind_{A^i}(\xi)}{N-1}$, just as $\mathcal{P}^i$ given in \eqref{eq:Pi}. As in \eqref{eq:dZi}, we define $\tilde{Z}_t^i$ as the solution to
\[
d\tilde{Z}_t^i
= \int_{\IR^2} \int_{[0,N)} [l\tilde{Z}_{t^-}^i + r F_t^i(\mathbf{Z}_{t^-},\xi) - \tilde{Z}_{t^-}^i] \tilde{\mathcal{P}}^i(dt,dl,dr,d\xi),
\]
with $\tilde{Z}_0^i = Z_0^i$. That is, $\tilde{Z}_t^i$ uses the same atoms of $\mathcal{P}$ that $Z_t^i$ uses, including the same $f_t$-distributed variables $F_t^i(\mathbf{Z}_{t^-},\xi)$, except for the joint jumps involving some particle $j\in\{1,\ldots,k\}$. In that case, either $\tilde{Z}_t^i$ or $\tilde{Z}_t^j$ does not jump at that instant; to compensate for the missing jumps, new, independent atoms, drawn from $\tilde{\mathcal{P}}$, are added to $\tilde{\mathcal{P}}^i$.

Since clearly $\tilde{\mathcal{P}}^1,\ldots,\tilde{\mathcal{P}}^k$ are i.i.d.\ Poisson point measures, it is straightforward to verify that $\tilde{Z}_t^1,\ldots,\tilde{Z}_t^k$ are independent nonlinear processes. Then,
\[
W_2^2(\law(Z_t^1,\ldots,Z_t^k), f_t^{\otimes k})
\leq \IE \frac{1}{k} \sum_{i=1}^k (Z_t^i - \tilde{Z}_t^i)^2
= g(t),
\]
for $g(t) = \IE [(Z_t^1 - \tilde{Z}_t^1)^2]$, thanks to exchangeability. Thus, it suffices to estimate $g(t)$:
\begin{equation}
\label{eq:dEZZ2}
\begin{split}
\frac{dg(t)}{dt}
&= \frac{d}{dt} \IE \int_0^t \int_{\IR^2} \int_{A^i}
  \Delta_s^1 [\mathcal{P}(ds,dl,dr,\IR,\IR,[i-1,i),d\xi) \\
& \qquad \qquad {} + \mathcal{P}(ds,\IR,\IR,dl,dr,d\xi,[i-1,i)) \ind_{[k,N)}(\xi) \\
& \qquad {} + \frac{d}{dt} \IE \int_0^t \int_{\IR^2} \int_{A^i}
  \Delta_s^2 \mathcal{P}(ds,\IR,\IR,dl,dr,d\xi,[i-1,i)) \ind_{[0,k)}(\xi) \\
& \qquad {} + \frac{d}{dt} \IE \int_0^t \int_{\IR^2} \int_{A^i}
  \Delta_s^3 \tilde{\mathcal{P}}(ds,\IR,\IR,dl,dr,d\xi,[i-1,i)) \ind_{[0,k)}(\xi), 
\end{split}
\end{equation}
where $\Delta_s^1$ is the increment of $(Z_s^1-\tilde{Z}_s^1)^2$ when $Z_s^1$ and $\tilde{Z}_s^1$ have a simultaneous jump, $\Delta_s^1$ is the increment when only $Z_s^1$ jumps, and $\Delta_s^3$ is the increment when only $\tilde{Z}_s^1$ jumps. Thanks to the indicator $\ind_{[0,k)}(\xi)$ and the uniform boundedness of the second moment of $f_t$ given by \eqref{eq:moments_ft}, the second and third terms in \eqref{eq:dEZZ2} are easily seen to be bounded by $Ck/N$, where $C>0$ is a constant that depends on the asserted quantities and may change from line to line. For the first integral in \eqref{eq:dEZZ2}, the fact that both $Z_s^1$ and $\tilde{Z}_s^1$ interact against the same $F_s^1(\mathbf{Z}_{s^-},\xi)$ gives rise to a contraction term, namely
\begin{align*}
\Delta_s^1
&= \left( [lZ_{s^-}^1 + rF_s^1(\mathbf{Z}_{s^-},\xi)]
- [l\tilde{Z}_{s^-}^1 + rF_s^1(\mathbf{Z}_{s^-},\xi)]\right)^2
- ( Z_{s^-}^1 - \tilde{Z}_{s^-}^1)^2 \\
&= - (1-l^2) ( Z_{s^-}^1 - \tilde{Z}_{s^-}^1)^2.
\end{align*}
Recall that the intensity of both $\mathcal{P}$ and is given by \eqref{eq:intensityP}. Thus, replacing this in \eqref{eq:dEZZ2}, writing $\ind_{[k,N)}(\xi) = \ind_{[0,N)}(\xi) - \ind_{[0,k)}(\xi)$ and bounding the $\ind_{[0,k)}(\xi)$ integral by $Ck/N$ as with the second and third terms, gives $\frac{dg(t)}{dt} \leq -\tilde{\aaa} g(t) + Ck/N$, for $\tilde{\aaa} = 1 - \frac{1}{2}\EE [L^2 + \tL^2]\geq \aaa >0$, thanks to \eqref{eq:as_preservation} and \eqref{eq:LRnot01}. Since $g(0) = 0$, using Gr\"onwall's lemma we easily deduce that $g(t) \leq Ck/N$, which proves \eqref{eq:W2Zkftk}.

To deduce the estimate of Lemma \ref{lem:decoupling} from \eqref{eq:W2Zkftk}, we need to recall two results. First, for any exchangeable random vector $\mathbf{X}$ on $\IR^N$ and any measure $\mu$ on $\IR$, both with finite second moment, using \cite[Lemma 7]{cortez-fontbona2016} we have
\begin{equation}
\label{eq:EW2Xmu}
\frac{1}{2} \IE W_2^2(\bar{\mathbf{X}}, \mu)
\leq W_2^2(\law(X^1,\cdots,X^k), \mu^{\otimes k}) + \varepsilon_k(\mu) + \frac{Ck}{N},
\end{equation}
where $C$ depends only on the second moments of $X^1$ and $\mu$. Here $\varepsilon_k(\mu)$ is defined as $\IE W_2^2(\bar{\mathbf{Y}},\mu)$, where $\mathbf{Y} = (Y^1,\ldots,Y^k)$ is a collection of i.i.d.\ random variables with common law $\mu$. Second, if $\mu$ has finite $q$ moment for some $q>4$, \cite[Theorem 1]{fournier-guillin2013} gives
\begin{equation}
\label{eq:epsk}
\varepsilon_k(\mu) \leq \frac{C}{k^{1/2}},
\end{equation}
for a constant $C$ depending only on $q$ and $\int |v|^q \mu(dv)$. Now: taking $\mathbf{X} = \mathbf{Z}_t$ and $\mu = f_t$, using \eqref{eq:W2Zkftk}, \eqref{eq:EW2Xmu} and \eqref{eq:epsk}, we have for any $k\leq N$,
\[
\frac{1}{2} \IE W_2^2(\bar{\mathbf{Z}}_t, f_t)
\leq W_2^2(\law(Z_t^1,\cdots,Z_t^k), f_t^{\otimes k}) + \varepsilon_k(f_t) + \frac{Ck}{N}
\leq \frac{C}{k^{1/2}} + \frac{Ck}{N},
\]
with $C$ depending on the uniform bound of $\int v^q f_t(dv)$ provided by \eqref{eq:moments_ft}. Taking $k \sim N^{2/3}$ gives $\IE W_2^2(\bar{\mathbf{Z}}_t, f_t) \leq CN^{-1/3}$, as desired. The estimate for $\bar{\mathbf{Z}}_t^1$ is deduced similarly.
\end{proof}

\begin{proof}[Proof of Lemma \ref{lem:Gronwall}]
The proof is almost the same as the one of \cite[Lemma 4.1.8]{ambrosio-gigli-savare2008}. Define $v(t)=u(t)^{1/2}$, thus $\frac{d}{dt}(v^2) \leq -av^2 + bv + c$. Multiplying by $e^{at}$ we obtain $\frac{d}{dt}(v^2 e^{at}) \leq bv e^{at} + ce^{at}$. Define $w(t) = v(t)e^{at/2}$ and $W(t) = \sup_{s\in[0,t]} w(t)$, thus $\frac{d}{dt}(w^2) \leq b w e^{at/2} + ce^{at}$. Integrating gives
\[
w^2(t) - w(0)^2
\leq b \int_0^t w(s) e^{as/2} ds + C(t)
\leq 2B(t)W(t) + C(t),
\]
for $B(t) = \frac{b}{a}e^{at/2}$ and $C(t) = \frac{c}{a}e^{at}$. Taking supremum on both sides gives $W^2(t) \leq w(0)^2 + 2B(t)W(t) + C(t)$. Adding $B(t)^2$ yields $(W(t)-B(t))^2 \leq w(0)^2 + B(t)^2 + C(t)$, which in turn gives $W(t) \leq B(t) + \sqrt{w(0)^2 + B(t)^2 + C(t)}$. Taking squares yields
\[
u(t)e^{at}
= w(t)^2
\leq W(t)^2
\leq 2w(0)^2 + 4B(t)^2 + 2 C(t).
\]
Multiplying by $e^{-at}$ gives the desired bound.
\end{proof}

\bibliographystyle{plain}
\bibliography{references.bib}{}

\end{document}